\documentclass[runningheads]{llncs}
\usepackage{graphicx}
\usepackage{url}

\usepackage{hyperref}

\usepackage{amssymb,amsmath, amsrefs}
\newtheorem{thm}{Theorem}

\newtheorem{cor}[thm]{Corollary}
\newtheorem{defi}[thm]{Definition}

\newtheorem{nota}[thm]{Notation}

\newcommand\be{\begin{equation}}
\newcommand\ee{\end{equation}} 

\usepackage{amsmath,amsfonts}

\def\bdefi{\begin{defi}\rm}
\def\edefi{\end{defi}}
\def\bnota{\begin{nota}\rm}
\def\enota{\end{nota}}

\def\ZFC{\textup{\textsf{ZFC}}}

\def\({\textup{(}}
\def\){\textup{)}}

\def\bye{\end{document}}
\def\N{{\mathbb  N}}
\def\Q{{\mathbb  Q}}
\def\R{{\mathbb  R}}

\def\SS{\textup{\textsf{S}}}

\def\di{\rightarrow}

\def\asa{\leftrightarrow}

\def\osc{\textup{\textsf{osc}}}

\def\eps{\varepsilon}

\usepackage{comment}

\begin{document}
\title{The non-normal abyss in Kleene's computability theory\thanks{This research was supported by the \emph{Deutsche Forschungsgemeinschaft} (DFG) (grant nr.\ SA3418/1-1) and the \emph{Klaus Tschira Boost Fund} (grant nr.\ GSO/KT 43).}}
\author{Sam Sanders\inst{1}}
\authorrunning{S.\ Sanders}
\institute{Department of Philosophy II, RUB Bochum, Germany \\
\email{sasander@me.com} \\
\url{https://sasander.wixsite.com/academic}}

\setcounter{secnumdepth}{3}
\setcounter{tocdepth}{3}

%\setcounter{page}{0}
%\tableofcontents
%\thispagestyle{empty}
%\newpage
%%
\maketitle              % typeset the header of the contribution
\begin{abstract}
Kleene's computability theory based on his S1-S9 computation schemes constitutes a model for \emph{computing with objects of any finite type} and extends Turing's `machine model' which formalises \emph{computing with real numbers}.  A fundamental distinction in Kleene's framework is between \emph{normal} and \emph{non-normal} functionals where the former compute the associated \emph{Kleene quantifier} $\exists^{n}$ and the latter do not.  Historically, the focus was on normal functionals, but recently new non-normal functionals have been studied, based on well-known theorems like the \emph{uncountability of the reals}. 
These new non-normal functionals are fundamentally different from historical examples like Tait's fan functional: the latter is computable from $\exists^{2}$ while the former are only computable in $\exists^{3}$.  
While there is a great divide separating $\exists^{2}$ and $\exists^{3}$, we identify certain closely related non-normal functionals that fall on different sides of this abyss.  
Our examples are based on mainstream mathematical notions, like \emph{quasi-continuity}, \emph{Baire classes}, and \emph{semi-continuity}.  
\end{abstract}
% 0) modulus of quasicontinuity is computable in \exists^{2}
%1) if a `modulus of cliquishness' is given, then we can compute x in C_{f}.
%2) in light of 1, a modulus of cliquisness is not computable by any SS_{k}
%3) 
\section{Introduction}

\subsection{Motivation and overview}\label{intro}
Computability theory is a discipline in the intersection of theoretical computer science and mathematical logic where the fundamental question is as follows:
\begin{center}
\emph{given two mathematical objects $X $ and $ Y$, does $X$ compute $Y$ {in principle}?}
\end{center} 
In case $X $ and $Y$ are real numbers, Turing's famous `machine' model (\cite{tur37}) is the standard approach to this question, i.e.\ `computation' is interpreted in the sense of Turing machines.  
To formalise computation involving (total) abstract objects, like functions on the real numbers or well-orderings of the reals, Kleene introduced his S1-S9 computation schemes in \cites{kleeneS1S9}.
%We assume basic familiarity with the latter as in \cite{longmann} while a sketch may be found in Section \ref{prelim}. 
Dag Normann and the author have recently introduced (\cite{dagsamXIII}) a version of the lambda calculus involving fixed point operators that exactly captures S1-S9 and accommodates partial objects. 
Henceforth, any reference to computability is to be understood in Kleene's framework and (if relevant) the extension from \cite{dagsamXIII}.

\smallskip

A fundamental distinction in Kleene's framework is between \emph{normal} and \emph{non-normal} functionals where the former compute the associated \emph{Kleene quantifier} $\exists^{n}$ and the latter do not (see Section \ref{kelim}).  Historically, the focus was on normal functionals in that only few examples of \emph{natural} non-normal functionals were even known.  The first such example was Tait's \emph{fan functional}, which computes a modulus of uniform continuity on input a continuous function on $2^{\N}$ (\cite{dagtait}).  

\smallskip

Recently, Dag Normann and the author have identified \emph{new} non-normal functionals based on mainstream theorems like e.g.\ the \emph{Heine-Borel theorem}, the \emph{Jordan decomposition theorem}, and the \emph{uncountability of $\R$}  (\cites{dagsamV,dagsamVII, dagsamIX, dagsamXII, dagsamX,dagsamXIII}).  These non-normal functionals are \emph{very different} as follows: Tait's fan functional is computable in $\exists^{2}$, making it rather tame; by contrast the following non-normal operation is not computable in any $\SS_{k}^{2}$, where the latter decides $\Pi_{k}^{1}$-formulas.
\be\label{ting}
\text{\emph{Given $Y:[0,1]\di \N$, find $x, y\in \R $ such that $x\ne_{\R} y$ and $Y(x)=_{\N}Y(y)$.}}
\ee
Clearly, this operation witnesses the basic fact there is no injection from the unit interval to the naturals.  
The operation in \eqref{ting} \emph{can} be performed by $\exists^{3}$, which follows from some of the many proofs that $\R$ is uncountable.
Essentially all the non-normal functionals studied in \cites{dagsamV,dagsamVII, dagsamIX, dagsamXII, dagsamX,dagsamXIII} compute the operation in \eqref{ting}, or some equally hard variation.

\smallskip

In light of the previous, there are two classes of non-normal functionals: those computable in $\exists^{2}$, like Tait's fan functional, and those computable \textbf{only} from $\exists^{3}$, like the operation in \eqref{ting}.  Given the difference in computational power between $\exists^{2}$ and $\exists^{3}$, there would seem to be a great divide between these two classes.
In this paper, we identify certain \emph{closely related} non-normal functionals that fall on different sides of this abyss.  In particular, we obtain the following results.  
\begin{itemize}
\item Basic operations (finding a point of continuity or the supremum) on \emph{quasi-continuous} functions can be done using $\exists^{2}$; the same operations on the closely related \emph{cliquish} functions are only computable in $\exists^{3}$ (Section \ref{diff}).
\item Finding the supremum of \emph{Baire 2} functions requires $\exists^{3}$; the same operation is computable in $\SS^{2}$ for \emph{effectively} Baire 2 functions (Section \ref{SC2}).

\item Basic operations (finding a point of continuity or the supremum) on \emph{semi-continuous} functions require $\exists^{3}$, even if we assume an oscillation function (Def.\ \ref{oscf}); the same operations are computable in $\exists^{2}$ if we assume a `modulus of semi-continuity' (Section \ref{SC}).
\end{itemize}
Finally, we briefly sketch Kleene's framework in Section \ref{prelim}.  Required axioms and definitions are introduced in Sections \ref{lll} and \ref{cdef}.

\subsection{Preliminaries and definitions}\label{kelim}
We briefly introduce Kleene's \emph{higher-order computability theory} in Section~\ref{prelim}.
We introduce some essential axioms (Section~\ref{lll}) and definitions (Section~\ref{cdef}).  A full introduction may be found in e.g.\ \cite{dagsamX}*{\S2} or \cite{longmann}.
Since Kleene's computability theory borrows heavily from type theory, we shall often use common notations from the latter; for instance, the natural numbers are type $0$ objects, denoted $n^{0}$ or $n\in \N$.  
Similarly, elements of Baire space are type $1$ objects, denoted $f\in \N^{\N}$ or $f^{1}$.  Mappings from Baire space $\N^{\N}$ to $\N$ are denoted $Y:\N^{\N}\di \N$ or $Y^{2}$. 
An overview of this kind of notations can be found in e.g.\ \cite{longmann, dagsamXIII}. 

%In Section \ref{crux}, we motivate our choice of definitions, Definition \ref{openset} in particular.  
\subsubsection{Kleene's computability theory}\label{prelim}
Our main results are in computability theory and we make our notion of `computability' precise as follows.  
\begin{enumerate}
\item[(I)] We adopt $\ZFC$, i.e.\ Zermelo-Fraenkel set theory with the Axiom of Choice, as the official metatheory for all results, unless explicitly stated otherwise.
\item[(II)] We adopt Kleene's notion of \emph{higher-order computation} as given by his nine clauses S1-S9 (see \cite{longmann}*{Ch.\ 5} or \cite{kleeneS1S9}) as our official notion of `computable' involving total objects.
\end{enumerate}
%Similar to \cites{dagsam,dagsamII, dagsamIII, dagsamV, dagsamVI, dagsamVII}, one main aim of this paper is the study of functionals of type 3 that are \emph{natural} from the perspective of mathematical practise. 
We mention that S1-S8 are rather basic and merely introduce a kind of higher-order primitive recursion with higher-order parameters. 
%\marginpar{\footnotesize{Better: a kind of primitive recursion with higher order parameters. (Not to confuse it with T)}}  
The real power comes from S9, which essentially hard-codes the \emph{recursion theorem} for S1-S9-computability in an ad hoc way.  
By contrast, the recursion theorem for Turing machines is derived from first principles in \cite{zweer}.

\medskip

On a historical note, it is part of the folklore of computability theory that many have tried (and failed) to formulate models of computation for objects of all finite type and in which one derives the recursion theorem in a natural way.  For this reason, Kleene ultimately introduced S1-S9, which 
were initially criticised for their aforementioned ad hoc nature, but eventually received general acceptance.  
Now, Dag Normann and the author have introduced a new computational model based on the lambda calculus in \cite{dagsamXIII} with the following properties:
\begin{itemize}
\item S1-S8 is included while the `ad hoc' scheme S9 is replaced by more natural (least) fixed point operators,
\item the new model exactly captures S1-S9 computability for total objects,
\item the new model accommodates `computing with partial objects',
\item the new model is more modular than S1-S9 in that sub-models are readily obtained by leaving out certain fixed point operators.
\end{itemize}
We refer to \cites{longmann, dagsamXIII} for a thorough overview of higher-order computability theory.
We do mention the distinction between `normal' and `non-normal' functionals  based on the following definition from \cite{longmann}*{\S5.4}. 
We only make use of $\exists^{n}$ for $n=2,3$, as defined in Section \ref{lll}.
%\marginpar{\footnotesize{A defintion should not be in \\emph-mode like a lemma or theorem?}}
%Correct observation. This was an odd behaviour of the defi environment, 
\bdefi\label{norma}
For $n\geq 2$, a functional of type $n$ is called \emph{normal} if it computes Kleene's quantifier $\exists^{n}$ following S1-S9, and \emph{non-normal} otherwise.  
\edefi
\noindent
It is a historical fact that higher-order computability theory, based on Kleene's S1-S9 schemes, has focused primarily on the world of \emph{normal} functionals; this opinion can be found \cite{longmann}*{\S5.4}.  
Nonetheless, we have previously studied the computational properties of new \emph{non-normal} functionals, namely those that compute the objects claimed to exist by:
\begin{itemize}
\item covering theorems due to Heine-Borel, Vitali, and Lindel\"of (\cites{dagsamV}),
\item the Baire category theorem (\cite{dagsamVII}),
\item local-global principles like \emph{Pincherle's theorem} (\cite{dagsamV}),
\item weak fragments of the Axiom of (countable) Choice (\cite{dagsamIX}),
\item the Jordan decomposition theorem and related results (\cites{dagsamXII, dagsamXIII}),
\item the uncountability of $\R$ (\cites{dagsamX, dagsamXI}).
\end{itemize}
Finally, the first example of a non-computable non-normal functional, Tait's fan functional (\cite{dagtait}), is rather tame: it is computable in $\exists^{2}$. 
By contrast, the functionals based on the previous list, including the operation \eqref{ting} from Section~\ref{intro}, are computable in $\exists^{3}$ but not computable in any $\SS_{k}^{2}$, where the later decides $\Pi_{k}^{1}$-formulas (see Section \ref{lll} for details).

\subsubsection{Some comprehension functionals}\label{lll}
In Turing-style computability theory, computational hardness is measured in terms of where the oracle set fits in the well-known comprehension hierarchy.  
For this reason, we introduce some axioms and functionals related to \emph{higher-order comprehension} in this section.
We are mostly dealing with \emph{conventional} comprehension here, i.e.\ only parameters over $\N$ and $\N^{\N}$ are allowed in formula classes like $\Pi_{k}^{1}$ and $\Sigma_{k}^{1}$.  

\medskip

First of all, the functional $\varphi^{2}$, also called \emph{Kleene's quantifier $\exists^{2}$}, as in $(\exists^{2})$ is clearly discontinuous at $f=11\dots$; in fact, $\exists^{2}$ is (computationally) equivalent to the existence of $F:\R\di\R$ such that $F(x)=1$ if $x>_{\R}0$, and $0$ otherwise via Grilliot's trick (see \cite{kohlenbach2}*{\S3}).
\be\label{muk}\tag{$\exists^{2}$}
(\exists \varphi^{2}\leq_{2}1)(\forall f^{1})\big[(\exists n)(f(n)=0) \asa \varphi(f)=0    \big]. 
\ee
Related to $(\exists^{2})$, the functional $\mu^{2}$ in $(\mu^{2})$ is called \emph{Feferman's $\mu$} (\cite{avi2}).
\begin{align}\label{mu}\tag{$\mu^{2}$}
(\exists \mu^{2})(\forall f^{1})\big(\big[ (\exists n)(f(n)=0) \di [f(\mu(f))=0&\wedge (\forall i<\mu(f))(f(i)\ne 0) \big]\\
& \wedge [ (\forall n)(f(n)\ne0)\di   \mu(f)=0] \big). \notag
\end{align}
We have $(\exists^{2})\asa (\mu^{2})$ over Kohlenbach's base theory (\cite{kohlenbach2}), while $\exists^{2}$ and $\mu^{2}$ are also computationally equivalent.  
Hilbert and Bernays formalise considerable swaths of mathematics using only $\mu^{2}$ in \cite{hillebilly2}*{Supplement IV}.
%$\RCAo$ and $\ACAo\equiv\RCAo+(\exists^{2})$ proves the same sentences as $\ACA_{0}$ by \cite{hunterphd}*{Theorem~2.5}. 

\medskip
\noindent
Secondly, the functional $\SS^{2}$ in $(\SS^{2})$ is called \emph{the Suslin functional} (\cite{kohlenbach2}).
\be\tag{$\SS^{2}$}
(\exists\SS^{2}\leq_{2}1)(\forall f^{1})\big[  (\exists g^{1})(\forall n^{0})(f(\overline{g}n)=0)\asa \SS(f)=0  \big].
\ee
%The system $\FIVE^{\omega}\equiv \RCAo+(\SS^{2})$ proves the same $\Pi_{3}^{1}$-sentences as $\FIVE$ by \cite{yamayamaharehare}*{Theorem 2.2}.   
By definition, the Suslin functional $\SS^{2}$ can decide whether a $\Sigma_{1}^{1}$-formula as in the left-hand side of $(\SS^{2})$ is true or false.   
We similarly define the functional $\SS_{k}^{2}$ which decides the truth or falsity of $\Sigma_{k}^{1}$-formulas.
%
%we also define 
%the system $\SIXK$ as $\RCAo+(\SS_{k}^{2})$, where  $(\SS_{k}^{2})$ expresses that $\SS_{k}^{2}$ exists.  
We note that the Feferman-Sieg operators $\nu_{n}$ from \cite{boekskeopendoen}*{p.\ 129} are essentially $\SS_{n}^{2}$ strengthened to return a witness (if existant) to the $\Sigma_{n}^{1}$-formula at hand.  %  if it exists. 

\medskip

\noindent
Thirdly, the following functional $E^{3}$ clearly computes $\exists^{2}$ and $\SS_{k}^{2}$ for any $k\in \N$:
\be\tag{$\exists^{3}$}
(\exists E^{3}\leq_{3}1)(\forall Y^{2})\big[  (\exists f^{1})(Y(f)=0)\asa E(Y)=0  \big].
\ee
%and we therefore define $\Z_{2}^{\Omega}\equiv \RCAo+(\exists^{3})$ and $\Z_{2}^\omega\equiv \cup_{k}\SIXK$, which are conservative over $\Z_{2}$ by \cite{hunterphd}*{Cor.\ 2.6}. 
%Despite this close connection, $\Z_{2}^{\omega}$ and $\Z_{2}^{\Omega}$ can behave quite differently, as discussed in e.g.\ \cite{dagsamIII}*{\S2.2}.   
The functional from $(\exists^{3})$ is also called \emph{Kleene's quantifier $\exists^{3}$}, and we use the same -by now obvious- convention for other functionals.  
Hilbert and Bernays introduce a functional $\nu^{3}$ in \cite{hillebilly2}*{Supplement IV}, and the latter is essentially $\exists^{3}$ which also provides a witness like $\nu_{k}$ does.

\medskip

In conclusion, the operation \eqref{ting} from Section \ref{intro} is computable in $\exists^{3}$ but not in any $\SS_{k}^{2}$, as established in \cite{dagsamXI}.
Many non-normal functionals exhibit the same `computational hardness' and we merely view this as support for the development of a separate scale for classifying non-normal functionals.    

\subsubsection{Some definitions}\label{cdef}
We introduce some definitions needed in the below, mostly stemming from mainstream mathematics.
We note that subsets of $\R$ are given by their characteristic functions (Definition \ref{char}), where the latter are common in measure and probability theory.
%In this paper, `continuity' refers to the usual `epsilon-delta' definition, well-known from the literature. 

\medskip
\noindent
First of all, we make use the usual definition of (open) set, where $B(x, r)$ is the open ball with radius $r>0$ centred at $x\in \R$.
\bdefi[Set]\label{char}~
\begin{itemize}
\item Subsets $A\subset \R$ are given by its characteristic function $F_{A}:\R\di \{0,1\}$, i.e.\ we write $x\in A$ for $ F_{A}(x)=1$ for all $x\in \R$.
\item A subset $O\subset \R$ is \emph{open} in case $x\in O$ implies that there is $k\in \N$ such that $B(x, \frac{1}{2^{k}})\subset O$.
\item A subset $C\subset \R$ is \emph{closed} if the complement $\R\setminus C$ is open. 
\end{itemize}
\edefi
\noindent
No computational data/additional representation is assumed in the previous definition.  
As established in \cites{dagsamXII, dagsamXIII}, one readily comes across closed sets in basic real analysis (Fourier series) that come with no additional representation. 
%The reader will find more motivation for our definition of open set in Section~\ref{dichtbij}.

\smallskip

Secondly, the following sets are often crucial in proofs in real analysis. 
\bdefi
The sets $C_{f}$ and $D_{f}$ respectively gather the points where $f:\R\di \R$ is continuous and discontinuous.
\edefi
One problem with $C_{f}, D_{f}$ is that the definition of continuity involves quantifiers over $\R$.  
In general, deciding whether a given $\R\di \R$-function is continuous at a given real, is as hard as $\exists^{3}$ from Section \ref{lll}.
For these reasons, the sets $C_{f}, D_{f}$ do exist in general, but are not computable in e.g.\ $\exists^{2}$.  We show that for quasi-continuous and semi-continuous functions, these sets are definable in $\exists^{2}$. 

\smallskip

Thirdly, to define $C_{f}$ using $\exists^{2}$, one can also (additionally) assume the existence of the oscillation function $\osc_{f}:\R\di \R$ as in Def.\ \ref{oscf}.  Indeed, the continuity of $f$ as $x\in \R$ is then equivalent to the \emph{arithmetical} formula $\osc_{f}(x)=_{\R}0$. 
\bdefi[Oscillation function]\label{oscf}
For any $f:\R\di \R$, the associated \emph{oscillation functions} are defined as follows: $\osc_{f}([a,b]):= \sup _{{x\in [a,b]}}f(x)-\inf _{{x\in [a,b]}}f(x)$ and $\osc_{f}(x):=\lim _{k \di \infty }\osc_{f}(B(x, \frac{1}{2^{k}}) ).$
\edefi
We note that Riemann and Hankel already considered the notion of oscillation in the context of Riemann integration (\cites{hankelwoot, rieal}).

\section{Quasi-continuity and related notions}\label{diff}
We study the notion of \emph{quasi-continuity} and the closely related concept of \emph{cliquishness}, as in Definition \ref{klop}. 
As discussed below, the latter is essentially the closure of the former under sums.  
Nonetheless, basic properties concerning quasi-continuity give rise to functionals computable in $\exists^{2}$ while the same functionals generalised to cliquish functions are not computable in any $\SS_{k}^{2}$ by Theorem \ref{plonkook2}.

\smallskip

First of all, Def.\ \ref{klop} has some historical background: Baire has shown that separately continuous $\R^{2}\di \R$ are \emph{quasi-continuous} in one variable;  he mentions in \cite{beren2}*{p.\ 95} that the latter notion (without naming it) was suggested by Volterra.  
\bdefi\label{klop} For $f:[0,1]\di \R$, we have the following definitions:
\begin{itemize}
\item $f$ is \emph{quasi-continuous} at $x_{0}\in [0, 1]$ if for $ \epsilon > 0$ and any open neighbourhood $U$ of $x_{0}$, 
there is open ${\emptyset\ne G\subset U}$ with $(\forall x\in G) (|f(x_{0})-f(x)|<\eps)$.
\item $f$ is \emph{cliquish} at $x_{0}\in [0, 1]$ if for $ \epsilon > 0$ and any open neighbourhood $U$ of $x_{0}$, 
there is a non-empty open ${ G\subset U}$ with $(\forall x, y\in G) (|f(x)-f(y)|<\eps)$.
\end{itemize}
\edefi
These notions have nice technical and conceptual properties, as follows.
\begin{itemize}
\item The class of cliquish functions is exactly the class of sums of quasi-continuous functions (\cite{bors, quasibor2, malin}). In particular, cliquish functions are closed under sums while quasi-continuous ones are not.
\item The pointwise limit (if it exists) of quasi-continuous functions, is always cliquish (\cite{holausco}*{Cor.\ 2.5.2}).
%\item Similar to Baire 1 functions, cliquish functions are \emph{exactly} the limits of lqco and uqco sequences (see \cite{ewert2}). 
\item The set $C_{f}$ is dense in $\R$ if and only if $f:\R\di \R$ is cliquish (see \cites{bors, dobo}).
\end{itemize}
Moreover, quasi-continuous functions can be quite `wild': if $\mathfrak{c}$ is the cardinality of $\R$, there are $2^{\mathfrak{c}}$ non-measurable quasi-continuous $[0,1]\di \R$-functions and $2^{\mathfrak{c}}$ measurable  quasi-continuous $[0,1]\di [0,1]$-functions (see \cite{holaseg}).  %Also, the class of quasi-continuous functions is closed under taking \emph{transfinite} limits (\cite{nieuwebronna}).  

\smallskip

Secondly, we show that $\exists^{2}$ suffices to witness basic properties of quasi-continuous functions. 
Hence, the associated functionals fall in the same class as Tait's {fan functional}.   We call a set `RM-open' if it is given via an RM-code (see \cite{simpson2}*{II.5.6}), i.e.\ a sequence of rational open balls.   
\begin{thm}\label{plonkook}
For quasi-continuous $f:[0,1]\di \R$, we have the following: 
\begin{itemize}
\item the set $C_{f}$ is definable using $\exists^{2}$ and the latter computes some $x\in C_{f}$,
\item there is a sequence $(O_{n})_{n\in \N}$ of RM-open sets, definable in $\exists^{2}$, such that $C_{f}=\cap_{n\in \N}O_{n}$,
\item the oscillation function $\osc_{f}:[0,1]\di \R$ is computable in $\exists^{2}$. 
\item the supremum $\sup_{x\in [p,q]}f(x)$ is computable in $\exists^{2}$ for any $p, q \in \Q\cap [0,1]$. 
\end{itemize}
\end{thm}
\begin{proof}
Fix quasi-continuous $f:[0,1]\di \R$ and use $\exists^{2}$ to define $x\in O_{m}$ in case 
\be\label{obvio}\textstyle
(\exists N_{0}\in \N)(\forall q, r\in  B(x, \frac{1}{2^{N_{0}}})\cap \Q)( |f(q)-f(r)|\leq \frac{1}{2^{m}}  ).  
\ee
By (the definition of) quasi-continuity, the formula \eqref{obvio} is equivalent to 
\be\label{ferengi}\textstyle
(\exists N_{1}\in \N)(\forall w, z\in  B(x, \frac{1}{2^{N_{1}}}))( |f(w)-f(z)|\leq \frac{1}{2^{m}}  ),
\ee
where we note that the equivalence remains valid if $N_{0}=N_{1}$ in \eqref{obvio} and \eqref{ferengi}.  
Now apply $\mu^{2}$ to \eqref{obvio} to obtain $G:([0,1]\times \N)\di \N$ such that for all $x\in [0,1]$ and $m\in \N$, we have
\[\textstyle
x\in O_{m}\di (\forall w, z\in  B(x, \frac{1}{2^{G(x, m)}}))( |f(w)-f(z)|\leq \frac{1}{2^{m}}  ).
\]
Hence, $x\in O_{m}\di B(x, \frac{1}{2^{G(x, m)}})\subset O_{m}$, witnessing that $O_{m}$ is open.  
Clearly, we also have $O_{m}=\cup_{q\in \Q}B(q, \frac{1}{2^{G(q, m)}})$, i.e.\ we also have an RM-representation of $O_{m}$.  
To find a point $x\in C_{f}= \cap_{m\in \N}O_{m}$, the proof of the Baire category theorem for RM-representations is effective by \cite{simpson2}*{II.5.8}, and the first two items are done.  
%Defining $D_{m}:= [0,1]\setminus O_{m}$, the first two items are done.  

\smallskip

For the final two items, note that $\sup_{x\in [p,q]}f(x)$ equals $\sup_{x\in [p,q]\cap \Q}f(x)$ due to the definition of quasi-continuity. 
In particular, in the usual interval-halving procedure for finding the supremum, one can equivalently replace `$(\exists x\in [0,1])(f(x)>y)$' by `$(\exists q\in [0,1]\cap \Q)(f(q)>y)$' in light of the definition of quasi-continuity.  
The same holds for infima and the oscillation function $\osc_{f}:[0,1]\di \R$ is therefore also computable in $\exists^{2}$.
\qed
\end{proof}

\smallskip

Thirdly, despite their close connection and Theorem \ref{plonkook}, basic properties of cliquish functions give rise to functionals that are \emph{hard} to compute in terms of comprehension functionals by Theorem \ref{plonkook2}.  To this end, we need the following definition from \cite{dagsamXIII}, which also witnesses that the unit interval is uncountable. 
\bdefi
Any $\Phi:\big( (\R\di \{0,1\})\times (\R\di \N)  \big)\di \R  $ is called a \emph{Cantor realiser} in case $\Phi(A, Y)\not \in A$ for non-empty $A\subset [0,1]$ and $Y:[0,1]\di \N$ injective on $A$. 
\edefi
As shown in \cite{dagsamXII}, no Cantor realiser is computable in any $\SS_{k}^{2}$, even if we require a bijection (rather than an injection). 
We have the following result.
\begin{thm}\label{plonkook2}
The following functionals are not computable in any $\SS_{k}^{2}$:
\begin{itemize}
\item any functional $\Phi:(\R\di \R)\di \R$ such that for all cliquish $f:[0,1]\di [0,1]$, we have $\Phi(f)\in C_{f}$. 
%the set $C_{f}$ is definable using $\exists^{2}$ and the latter computes some $x\in C_{f}$,
\item  any functional $\Psi:(\R\di \R)\di (\R^{2}\di \R)$ such that for cliquish $f:[0,1]\di [0,1]$, we have $\Psi(f, p, q)=\sup_{x\in [p,q]}f(x)$ for $p, q\in [0,1]$. 
%the supremum $\sup_{x\in [p,q]}f(x)$ is computable in $\exists^{2}$ for any $p, q \in \Q\cap [0,1]$. 
\item  any functional $\zeta:(\R\di \R)\di ( (\N\times \N)\di \Q^{2})$ such that for cliquish $f:[0,1]\di [0,1]$ and any $n, m\in \N$, $\zeta(f, m,n)$ is an open interval such that $C_{f}=\bigcap_{n\in \N}\big(\cup_{m\in \N}\zeta(f, m,n) \big)$.
%there is a sequence $(O_{n})_{n\in \N}$ of RM-open sets such that $C_{f}=\cap_{n\in \N}O_{n}$,
%the oscillation function $\osc_{f}:[0,1]\di \R$ is computable in $\exists^{2}$. 
\end{itemize}
In particular, each of these functionals computes a Cantor realiser \(given $\exists^{2}$\).  
\end{thm}
\begin{proof}
Fix $A\subset [0,1]$ and $Y:[0,1]\di \N$ injective on $A$.  Now define the following function $f:[0,1]\di \R$, for any $x\in [0,1]$, as follows:
\be\label{penny}
f(x):=
\begin{cases}
\frac{1}{2^{Y(x)+1}} & \textup{in case $x\in A$}\\
0 & \textup{otherwise}
\end{cases}.
\ee
By definition, for any $\eps>0$, there are only finitely many $x\in A$ such that $f(x)>\eps$ for $i\leq k$.  
This readily implies that $f$ is \emph{cliquish} at any $x\in [0,1]$ and \emph{continuous} at any $y\not \in A$.  
Now let $\Phi$ be as in the first item and note that $\Phi(f)\in C_{f}$ implies that $\Phi(f)\not \in A$, as required for a Cantor realiser.  

\smallskip

For the second item, let $\Psi$ be as in the latter and consider $\Psi(f, 0,1)$, which has the form $\frac{1}{2^{n_{0}+1}}=f(y_{0})$ for some $y_{0}\in [0,1]$ and $n_{0}\in \N$. 
Now check whether $\Psi(f, 0, \frac{1}{2})= \Psi(f, 0, 1)$ to decide if $y_{0}\in [0,\frac12]$ or not.  Hence, we know the first bit of the binary representation of $y_{0}$.
Repeating this process, we can \emph{compute} $y_{0}$, and similarly obtain an enumeration of $A$.  With this enumeration, we can compute $z\not \in A$ following \cite{simpson2}*{II.4.9}, as required for a Cantor realiser. 

\smallskip

For the third item, to find a point $x\in C_{f}=\bigcap_{n\in \N}\big(\cup_{m\in \N}\zeta(f, m,n) \big)$, the proof of the Baire category theorem for RM-representations is effective by \cite{simpson2}*{II.5.8}, and the first item provides a Cantor realiser.  
\qed
\end{proof}
%As a more advanced result, we now consider a version of the Heine-Borel theorem.  (space permitting)

\section{The first and second Baire classes}\label{SC2}
We study the notion of \emph{Baire 1} function and the closely related concept of (effectively) \emph{Baire 2} function, as in Definition \ref{flung}. 
Nonetheless, basic properties of Baire 1 functions give rise to functionals computable in $\exists^{2}$ while the same functionals generalised to Baire 2 are not computable in any $\SS_{k}^{2}$.
Properties of \emph{effectively} Baire 2 functions are still computable by the Suslin functional $\SS^{2}$.

\smallskip

First of all, after introducing the Baire classes, Baire notes that Baire 2 functions can be \emph{represented} by repeated limits as in \eqref{kabel} (see \cite{beren2}*{p.\ 69}). 
Given $\exists^{2}$, effectively Baire 2 functions are essentially the representation of Baire~2 functions used in second-order arithmetic (\cite{basket2}).
\bdefi\label{flung} 
For $f:[0,1]\di \R$, we have the following definitions:
\begin{itemize}
\item $f$ is \emph{Baire $1$} if it is the pointwise limit of a sequence of continuous functions.
\item $f$ is \emph{Baire $2$} if it is the pointwise limit of a sequence of Baire 1 functions. 
\item $f$ is \emph{effectively Baire $2$} if there is a double sequence $(f_{n,m})_{n,m\in \N}$ of continuous functions on $[0,1]$ such that
\be\label{kabel}\textstyle
f(x)=_{\R}\lim_{n\di \infty }\lim_{m\di \infty}f_{n,m}(x) \textup{ for all $x\in [0,1]$}.
\ee

\end{itemize}
\edefi
Secondly, the following theorem -together with Theorem \ref{banks}- shows there is a great divide in terms of computability theoretic properties for Baire 2 functions and representations.  
Note that for effectively Baire 2 functions, we assume the associated (double) sequence is an input for the algorithm.  
\begin{thm}\label{poi}
For effectively Baire 2 $f:[0,1]\di [0,1]$, the supremum $\sup_{x\in [p,q]}f(x)$ is computable in $\SS^{2}$ for any $p, q \in \Q\cap [0,1]$. 
\end{thm}
\begin{proof}
Let $(f_{n,m})$ be a double sequence as in \eqref{kabel}.  By the definition of repeated limit, the formula $(\exists x\in [0,1])(f(x)>q)$ is equivalent to 
\[\label{tanker}\textstyle
(\exists y\in [0,1], l\in \N)(\exists N\in \N)(\forall n\geq N)(\exists M\in \N)(\forall m\geq M)(f_{n,m}(y)\geq q +\frac{1}{2^{l}} ),
\]
which is equivalent to a $\Sigma_{1}^{1}$-formula upon replacing $f_{n, m}$ by RM-codes codes for continuous functions.
Note that $\exists^{2}$ computes such codes (uniformly) by \cite{kohlenbach4}*{\S4} (for Baire space) and \cite{dagsamXIV}*{\S2.2} (for $\R$).  
In light of the above equivalence, $\SS^{2}$ can decide $(\exists x\in [0,1])(f(x)>q)$ and hence compute the required suprema.  
\qed
\end{proof}
By the results in \cite{dagsamXIV}*{\S2.3.1}, $\exists^{2}$ can compute the supremum of a bounded Baire~1 function.  
One could explore similar results for sub-classes.  

\smallskip

Thirdly, we have the following theorem.  Note that for Baire 2 functions, we assume the associated sequence of Baire 1 functions is an input for the algorithm.  
\begin{thm}\label{banks}
The following functionals are not computable in any $\SS_{k}^{2}$:
\begin{itemize}
\item  any functional $\Phi:(\R\di \R)\di (\R^{2}\di \R)$ such that for Baire 2 $f:[0,1]\di [0,1]$, we have $\Phi(f, p, q)=\sup_{x\in [p,q]}f(x)$ for $p, q\in [0,1]$. 
\item  any functional $\Psi:(\R\di \R)\di (\N^{2}\di (\R\di \R))$ such that for Baire 2 $f:[0,1]\di [0,1]$, the double sequence $(\Psi(f, n, m))_{n,m\in \N}$ satisfies \eqref{kabel}.  % for $x\in [0,1]$. 
\end{itemize}
In particular, each of these functionals computes a Cantor realiser \(given $\SS^{2}$\).  
\end{thm}
\begin{proof}
For the first item, $f$ as in \eqref{penny} is Baire 2.  Indeed, consider the following
\be\label{penny2}
f_{n}(x):=
\begin{cases}
\frac{1}{2^{Y(x)+1}} & \textup{in case $x\in A \wedge Y(x)\leq n$}\\
0 & \textup{otherwise}
\end{cases},
\ee
which has only got at most $n+1$ points of discontinuity, i.e.\ $f_{n}$ is definitely Baire~1. 
We trivially have $\lim_{n\di \infty}f_{n}(x)=f(x)$ for $x\in [0,1]$. 
For the second item, combine the results for the first item with Theorem \ref{poi}. 
\qed
\end{proof}

\section{Semi-continuity}\label{SC}
We study the notion of \emph{upper and lower semi-continuity} due to Baire (\cite{beren2}).  
Curiously, we \emph{can} define $C_{f}$ for a usco $f:[0,1]\di \R$ using $\exists^{2}$, but computing an $x\in C_{f}$ is \emph{not} possible via any $\SS_{k}^{2}$ (see Theorems \ref{timtam} and \ref{tam}), even assuming an oscillation function.  
Requiring a `modulus of semi-continuity' (see Def.\ \ref{klung}), $\exists^{2}$ can compute some $x\in C_{f}$ (Theorem \ref{tim}).  
However, while a modulus of \emph{continuity} is computable in $\exists^{2}$, a modulus of \emph{semi-continuity} is not computable in any $\SS_{k}^{2}$ by Corollary~\ref{corke}.  
%Nonetheless, basic properties concerning quasi-continuity give rise to functionals computable in $\exists^{2}$ 
%while the same functionals generalised to cliquish functions are not computable in any $\SS_{k}^{2}$ by Theorem \ref{plonkook2}.

\smallskip

\noindent
First of all, we use the following standard definitions.  
%we study semi-continuity due to Baire (\cite{beren2}), as follows. % (see e.g.\ \cites{nieuwebron, kowalski}).
\bdefi[Semi-continuity]\label{klung} 
For $f:[0,1]\di \R$, we have the following:
\begin{itemize}
\item $f$ is \emph{upper semi-continuous} (usco) at $x_{0}\in [0,1]$ if for any $y>f(x_{0})$, there is $N\in \N$ such that for all $z\in B(x, \frac{1}{2^{N}})$, we have $f(z)<y$,
% $f(x_{0})\geq_{\R}\lim\sup_{x\di x_{0}} f(x)$,
\item $f$ is \emph{lower semi-continuous} (lsco) at $x_{0}\in [0,1]$ if for any $y<f(x_{0})$, there is $N\in \N$ such that for all $z\in B(x, \frac{1}{2^{N}})$, we have $f(z)>y$,
\item a \emph{modulus of usco} for $f$ is any function $\Psi:[0,1]\di \R^{+}$ such that :
\[\textstyle
(\forall k\in \N) (\forall y\in B(x, \Psi(x,k)))( f(y)< f(x)+\frac{1}{2^{k}}   ).
\]  
We also refer to $\Psi$ as a `usco modulus'.  
\end{itemize}
%In case the weak continuity notion is satisfied at each point of the domain, the associated function satisfies the italicised notion. 
\edefi
Secondly, we have the following theorem.  
\begin{thm}\label{timtam}
For usco $f:[0,1]\di \R$,  the set $C_{f}$ is definable using $\exists^{2}$.
\end{thm}
\begin{proof}
First of all, it is a matter of definitions to show the equivalence between `$g:\R\di \R$ is continuous at $x\in \R$' and `$g:\R\di \R$ is usco and lsco at $x\in \R$'.  
Then, for usco $f:[0,1]\di \R$, `$f$ is discontinuous at $x\in [0,1]$' is equivalent to 
\be\label{dumbel}\textstyle
(\exists l\in \N)(\forall k\in \N){(\exists y\in B(x, \frac{1}{2^{k}})}(f(y)\leq f(x)-\frac{1}{2^{l}} ),
\ee
which expresses that $f$ is not lsco at $x\in [0,1]$.  Now, \eqref{dumbel} is equivalent to 
\be\label{dumbel1}\textstyle
(\exists l\in \N)(\forall k\in \N)\underline{(\exists r\in B(x, \frac{1}{2^{k}})\cap \Q)}(f(r)\leq f(x)-\frac{1}{2^{l}} ),
\ee
where in particular the underlined quantifier in \eqref{dumbel1} has rational range due to $f$ being usco.  
Since \eqref{dumbel1} is arithmetical, $\exists^{2}$ allows us to define $D_{f}$ (and $C_{f}$).  % and we are done. 
\qed
\end{proof}
Thirdly, we have the following theorem showing that while $C_{f}$ is definable using $\exists^{2}$, the latter cannot compute any $x\in C_{f}$ (and the same for any $\SS_{k}^{2}$), even if we assume an oscillation function (see Def.\ \ref{oscf}).  
\begin{thm}\label{tam}
Theorem \ref{plonkook2} remains correct if we replace `cliquish' by `usco' or `usco with an oscillation function'.
\end{thm}
\begin{proof}
The function $f$ from \eqref{penny} is usco, which follows from the observation that for any $\eps>0$, there are only finitely many $x\in A$ such that $f(x)>\eps$ for $i\leq k$.  
Now repeat the proof of Theorem \ref{plonkook2} for usco functions.  One readily proves that $f$ equals $\osc_{f}$, i.e.\ $f$ is its own oscillation function. 
\qed
\end{proof}
To our surprise, functions that are their own oscillation function are studied in the mathematical literature (\cite{kosten}).  Moreover, there is no contradiction between Theorems~\ref{poi} and~\ref{tam} as follows: 
while usco functions are Baire 1, Theorem \ref{tam} does not assume a Baire 1 (or effectively Baire 2) representation is given as an input, while of course Theorem \ref{poi} does.  

\smallskip

\noindent
Fourth, we now show that given a modulus of usco, we can find points of continuity of usco functions using $\exists^{2}$. 
\begin{thm}\label{tim}
For usco $f:[0,1]\di \R$ with a modulus $\Psi:[0,1]\di \R^{+}$, a real $x\in C_{f}$ can be computed by $\exists^{2}$. 
\end{thm}
\begin{proof}
Fix usco $f:[0,1]\di \R$ with modulus $\Psi:[0,1]\di \R^{+}$ and note that for $x\in [0,1]$ and $q\in \Q$, we have by definition that:
\[\textstyle
 (\exists N\in \N)(\forall z\in B(x, \frac{1}{2^{N}}))( f(z)\geq  q  ) \asa  (\exists M\in \N)\underline{(\forall r\in B(x, \frac{1}{2^{M}})\cap \Q)}( f(r)\geq q  ),
\]
where we abbreviate the right-hand side (arithmetical) formula by $A(x, q)$.  We note that the above equivalence even goes through for $N=M$.  
Define $O_{q}:=\{ x\in [0,1]: f(x)< q \vee A(x, q)\}$ using $\exists^{2}$ and note that $D_{q}:= [0,1]\setminus O_{q}$ is closed and (by definition) nowhere dense.  

\smallskip

Next, we show that $D_{f}\subset \cup_{q\in \Q}D_{q}$.  
Indeed, in case $x_{0}\in D_{f}$, $f$ cannot be lsco at $x_{0}\in [0,1]$, i.e.\ we have 
\be\label{tingel}\textstyle
(\exists l\in \N)(\forall N\in \N)(\exists z\in B(x_{0}, \frac{1}{2^{N}}))(f(z)\leq f(x_{0})-\frac{1}{2^{l}}).
\ee
Let $l_{0}$ be as in \eqref{tingel} and consider $q_{0}\in\Q$ such that $f(x_{0})>q_{0}> f(x_{0})-\frac{1}{2^{l_{0}}}$. 
By definition, $f(x_{0})\geq q_{0}$ and $\neg A(x_{0}, q_{0})$, i.e.\ $x_{0}\in D_{q_{0}}$ as required.  

\smallskip

Finally, define $Y(x)$ as $\Psi(x, k_{0} )$ in case $k_{0}$ is the least $k\in \N$ with $f(x)+\frac{1}{2^{k}}  \leq q$ (if such exists), and zero otherwise. 
In case $x \in O_{q} \wedge f(x)<q$, then $B(x, Y(x))\subset O_{q}$.  In case $x\in O_{q} \wedge A(x, q)$, then $\mu^{2}$ can find $M_{0}$, the least $M\in \N$ as in $A(x, q)$, which is such that $B(x, \frac{1}{2^{M_{0}}})\subset O_{q}$.  Hence, in case $x\in O_{q}$, we can compute (using $\mu^{2}$) some ball around $x$ completely within $O_{q}$.   The latter kind of representation of open sets is called the \emph{R2-representation} in \cite{dagsamVII}.  
Now, the Baire category theorem implies that there exists $y\in \cup_{q\in \Q}O_{q}$, which satisfies $y\not \in D_{f}$ by the previous paragraph.  
By \cite{dagsamVII}*{Theorem 7.10}, $\exists^{2}$ can compute such $y\in \cup_{q\in \Q}O_{q}$, thanks to the R2-representation of open sets.
Essentially, the well-known constructive proof goes through (see e.g.\ \cite{bish1}*{p.\ 87}) and one uses the R2-representation to avoid the use of the (countable) Axiom of Choice.     
\qed
\end{proof}
The following corollary should be contrasted with the fact that a modulus of continuity for real functions is computable from $\exists^{2}$.  
\begin{cor}\label{corke}
The following functional is not computable in any $\SS_{k}^{2}$:
\begin{center}
any functional $\Phi:(\R\di \R)\di ((\R\times \N)\di \R)$ such that $\Phi(f)$ is a usco modulus for usco $f:[0,1]\di [0,1]$. 
\end{center}
\end{cor}
\begin{proof}
Combine Theorems \ref{tam} and \ref{tim}. 
\qed
\end{proof}

%\begin{ack}\rm
%We thank Anil Nerode for his valuable advice and discussions related to this topic.
%\end{ack}

\section*{Bibliography}
%\begin{bibdiv}

%\end{bibdiv}

\bye